\documentclass[pdflatex,sn-mathphys-num]{sn-jnl}% Math and Physical Sciences Numbered Reference Style
\usepackage{graphicx}%
\usepackage{multirow}%
\usepackage{amsmath,amssymb,amsfonts}%
\usepackage{amsthm}%
\usepackage{mathrsfs}%
\usepackage[title]{appendix}%
\usepackage{xcolor}%
\usepackage{textcomp}%
\usepackage{manyfoot}%
\usepackage{booktabs}%
\usepackage{algorithm}%
\usepackage{algorithmicx}%
\usepackage{algpseudocode}%
\usepackage{listings}%
\usepackage{tabularx}
\AtBeginDocument{
  \setlength{\abovedisplayskip}{8pt plus 2pt minus 2pt}
  \setlength{\belowdisplayskip}{8pt plus 2pt minus 2pt}
  \setlength{\abovedisplayshortskip}{8pt plus 2pt minus 2pt}
  \setlength{\belowdisplayshortskip}{8pt plus 2pt minus 2pt}}
\theoremstyle{thmstyleone}%
\newtheorem{theorem}{Theorem}%  meant for continuous numbers
\newtheorem{corollary}{Corollary}
\newtheorem{lemma}[theorem]{Lemma}

\theoremstyle{thmstyletwo}%
\newtheorem{remark}{Remark}%

\theoremstyle{thmstylethree}%
\newtheorem{definition}{Definition}%

\begin{document}

\title[Stability and Error Estimates for Convolution Quadrature under Radial Symbol Bounds]{Stability and Error Estimates for Convolution Quadrature under Radial Symbol Bounds}

%%=============================================================%%
%% GivenName	-> \fnm{Joergen W.}
%% Particle	-> \spfx{van der} -> surname prefix
%% FamilyName	-> \sur{Ploeg}
%% Suffix	-> \sfx{IV}
%% \author*[1,2]{\fnm{Joergen W.} \spfx{van der} \sur{Ploeg}
%%  \sfx{IV}}\email{iauthor@gmail.com}
%%=============================================================%%

\author[1]{\fnm{Jingchuang} \sur{Ren}}\email{ren\_math@mail.tsinghua.edu.cn}
\equalcont{These authors contributed equally to this work.}

\author*[1]{\fnm{Haibin} \sur{Su}}\email{hbsu@mail.tsinghua.edu.cn}
\equalcont{These authors contributed equally to this work.}

\affil[1]{\orgdiv{Yau Mathematical Sciences Center}, \orgname{Tsinghua University}, \orgaddress{\city{Beijing}, \postcode{100084}, \country{China}}}

%%==================================%%
%% Sample for unstructured abstract %%
%%==================================%%

\abstract{
We develop stability and error estimates for multistep convolution
quadrature (CQ) under general radial bounds on the Laplace symbol.
Instead of the classical single-power condition $|F(z)|\leq M|z|^{-\mu}$, we assume $|F(z)|\leq C\Phi(|z|)$, where $\Phi$ is positive, continuous and regularly varying at zero and infinity. For strongly $A(\alpha)$-stable multistep methods, we derive error estimates and stability estimates for the CQ weights, as well as error bounds for power-type inputs $g_\beta(t)=t^{\beta-1}/\Gamma(\beta)$ with $\beta>0$. These estimates explicitly retain the dependence on the memory law through the factor $\Phi(t_n^{-1})$. The framework accommodates symbols whose behavior cannot be described by a global single-power. It also preserves logarithmic and multiscale factors that would be lost under a coarser power bounds. Then we apply the to fractional viscoelastic, distributed-order, and logarithmic-power memory models. Numerical experiments with backward Euler and second-order backward differentiation CQ confirm the expected temporal convergence orders for direct responses. A
distributed-order PDE further illustrates temporal convergence for smooth and nonsmooth initial data.
}


%%\pacs[JEL Classification]{D8, H51}


\maketitle

\section{Introduction}\label{sec:introduction}

Convolution quadrature (CQ) provides a systematic way to approximate
fractional convolutions from their Laplace transforms. For a kernel
$f$ with Laplace transform $F$, the continuous operator takes
the form
\[
F(\partial_t)g(t)
:=
\int_0^t f(t-s)g(s)\,\mathrm ds.
\]
On the uniform grid $t_n=n\tau$ with step size $\tau$, a multistep CQ is defined by
\[
F(\partial_\tau)g(t_n)
=
\sum_{j=0}^n \omega_j g(t_{n-j}),
\qquad
\sum_{j=0}^\infty \omega_j\xi^j
=
F\!\left(\frac{\delta(\xi)}{\tau}\right),
\]
where $\omega_j$ are the CQ weights and $\delta$ is the generating function of a linear multistep method. The construction goes back to Lubich's discretized fractional calculus and operational calculus, and was subsequently developed for general sectorial Laplace transforms and, more generally, for half-plane settings. Its main advantages include the treatment of singular kernels, kernels with several time scales, and problems for which the Laplace transform is available more readily than the time-domain kernel; see, e.g., \cite{lubich1986discretized, lubich1988convolution, lubich1988convolutionII, lubich1992convolution, lubich2004convolution}.

An important condition in the classical sectorial analysis is a
polynomial bound on the Laplace symbol. In the notation of
\cite{lubich2004convolution}, one assumes analyticity in a shifted sector
together with
\[
|F(z)|\le M|z|^{-\mu}
\]
for some $\mu\in\mathbb R$. Combined with the stability and consistency properties of the underlying multistep method, such a bound yields estimates for the CQ weights and for the CQ approximation of power-type inputs.

For some memory symbols, however, a single power does not accurately
represent the radial behavior of the complete Laplace symbol. Typical
examples include
\[
c_1z^{-a}+c_2z^{-b},
\qquad
\int_0^1 z^{-q}\,\mathrm dq,
\qquad
-z^{-a}\log z,
\]
with $0<a<b<1$. These symbols arise, respectively, in fractional
viscoelasticity \cite{adolfsson2004adaptive, lopez2008adaptive, bonfanti2020fractional, makris2020time},
distributed-order models \cite{jin2016error,yin2021approximation},
and fractional-order sensitivity analysis
\cite{kharazmi2019fractional}. Related logarithmic-power kernels are
studied in \cite{singh2024weakly}. Their behavior may involve different power scales at zero and infinity or additional slowly varying factors. Although individual examples can often be treated by shifted-sector arguments, decomposition into power symbols, or problem-specific estimates, such treatments do not in general retain the full radial dependence of the complete symbol.

The aim of this paper is to establish stability and error estimates for multistep convolution quadrature directly under radial symbol bounds of the form
\[
|F(z)|\le C\Phi(|z|),
\]
where $\Phi:(0,\infty)\to(0,\infty)$ is continuous and regularly varying at both zero and infinity, possibly with different indices. These regular variation assumptions yield a global Potter estimate. This enables a unified contour integral analysis that preserves the radial scale in both short- and long-time error estimates.

Let $f_n=\omega_n/\tau$ be the scaled CQ weight. For a
strongly $A(\alpha)$-stable multistep method of order $p$, we prove
\[
|f_n-f(t_n)|\le C\tau^pt_n^{-p-1}\Phi(t_n^{-1}),
\qquad
|f_n|+|f(t_n)|\le Ct_n^{-1}\Phi(t_n^{-1}).
\]
These estimates hold for every $t_n=n\tau>0$, with constants
independent of $\tau$, $n$ and a terminal time. For the power-type input $g_\beta(t)=t^{\beta-1}/\Gamma(\beta)$, $\beta>0$, we further obtain
\[
\|F(\partial_\tau)g_\beta(t_n)-F(\partial_t)g_\beta(t_n)\|
\le C\tau^{\min\{p,\beta\}}
t_n^{\beta-\min\{p,\beta\}-1}\Phi(t_n^{-1}).
\]
These estimates extend the corresponding power-bound results while retaining the radial scale $\Phi(t_n^{-1})$ explicitly.

We illustrate the framework with fractional Maxwell and Kelvin–Voigt models, distributed-order memory, and a logarithmic-power symbol obtained by differentiating with respect to the fractional order. The radial bounds retain the contributions of multiple fractional orders and the logarithmic factors present in these examples. Numerical experiments with backward Euler and the second-order backward differentiation formula confirm the predicted temporal orders and short-time scales, including for a distributed-order PDE with smooth and nonsmooth initial data.

The remainder of the paper is organized as follows. Section~\ref{sec:CQ-preliminaries} recalls the multistep CQ construction and the operational identities used in the analysis. Section~\ref{sec:regularly-varying-majorants} establishes the kernel and power-type input estimates under the assumption of regular variation at both endpoints. Section~\ref{sec:applications} presents the applications to fractional viscoelasticity and memory operators, Section~\ref{sec:numerical-experiments} shows the numerical experiments, and Section~\ref{sec:conclusion} concludes the paper.

\section{Preliminaries}
\label{sec:CQ-preliminaries}

We recall the construction of multistep CQ and the associated operational viewpoint introduced by Lubich \cite{lubich1988convolution,lubich1988convolutionII,
lubich2004convolution}. The derivation shows that CQ can be obtained
by applying a linear multistep method to a family of parameter-dependent
ordinary differential equations arising from the Laplace representation
of the convolution. This viewpoint identifies the two main ingredients
in the CQ analysis: the consistency and stability properties of the
underlying multistep method, and suitable bounds on the Laplace symbol
at the complex frequencies generated by the discretization.

\subsection{Construction and stability of multistep CQ}
\label{subsec:CQ-construction}

Let $X$ be a complex Banach space and let $\mathcal{L}(X)$ denote
the bounded linear operators on $X$. We set
\[
\Sigma_{\pi-\theta}
=
\{z\in\mathbb C\setminus\{0\}:|\arg z|<\pi-\theta\},
\qquad
0<\theta<\frac{\pi}{2},
\]
and use the principal branch for complex powers and logarithms.

Let $f$ be a kernel with Laplace transform $F$. For an input $g$, the continuous response is written as
\begin{equation}
\label{eq:CQ-continuous}
F(\partial_t)g(t)
:=
(f*g)(t)
=
\int_0^t f(t-s)g(s)\,\mathrm ds.
\end{equation}
To explain the construction of CQ, assume temporarily that $F$ is
analytic in $\Sigma_{\pi-\theta}$ and satisfies the classical decay
condition
\[
|F(z)|\le M|z|^{-\mu},
\qquad
\mu>0,
\]
so that its inverse Laplace transform is locally integrable.
Using Laplace inversion and interchanging the order of integration,
\eqref{eq:CQ-continuous} can be represented as
\begin{equation}
\label{eq:CQ-contour}
F(\partial_t)g(t)
=
\frac{1}{2\pi\mathrm i}
\int_\Gamma F(\lambda)y(t,\lambda)\,\mathrm d\lambda,
\end{equation}
where
\[
y(t,\lambda)
=
\int_0^t e^{\lambda(t-s)}g(s)\,\mathrm ds
\]
solves
\begin{equation}
\label{eq:CQ-parametric-ode}
\partial_t y
=
\lambda y+g,
\qquad
y(0,\lambda)=0.
\end{equation}
Thus the convolution is represented as a superposition of
parameter-dependent evolution problems in the Laplace domain.

Let $t_n=n\tau$ and apply a $k$-step linear multistep method to
\eqref{eq:CQ-parametric-ode}. Set
\[
\delta(\xi)
=
\frac{\sum_{j=0}^k a_j\xi^j}
     {\sum_{j=0}^k b_j\xi^j},
\]
and
\[
Y(\xi,\lambda)=\sum_{n=0}^{\infty}y_n(\lambda)\xi^n,
\qquad G(\xi)=\sum_{n=0}^{\infty}g(t_n)\xi^n,
\]
the multistep scheme yields
\[
Y(\xi,\lambda)
=
\left(
\frac{\delta(\xi)}{\tau}-\lambda
\right)^{-1}
G(\xi).
\]
Substitution into \eqref{eq:CQ-contour} and Cauchy's formula give
\[
\frac{1}{2\pi\mathrm i}
\int_\Gamma
F(\lambda)
\left(
\frac{\delta(\xi)}{\tau}-\lambda
\right)^{-1}
\,\mathrm d\lambda
=
F\left(\frac{\delta(\xi)}{\tau}\right).
\]
This leads to the CQ weights
\begin{equation}
\label{eq:CQ-weights}
F\left(\frac{\delta(\xi)}{\tau}\right)
=
\sum_{j=0}^{\infty}\omega_j\xi^j,
\end{equation}
and hence to
\begin{equation}
\label{eq:CQ-discrete}
F(\partial_\tau)g(t_n)
=
\sum_{j=0}^n
\omega_j g(t_{n-j}).
\end{equation}
Accordingly, CQ can be interpreted as a discretized operational
calculus in which the Laplace variable $z$ is replaced by the discrete
differentiation symbol $\delta(\xi)/\tau$.

We assume throughout that the multistep method is strongly
$A(\alpha)$-stable and of order $p$ in the sense of
\cite[Eq.~(2.5)]{lubich2004convolution}, with $\alpha>\theta$.
In particular,
\begin{equation}
\label{eq:CQ-multistep-condition}
|\arg\delta(\xi)|
\le \pi-\alpha,
\qquad |\xi|<1,
\end{equation}
and
\begin{equation}
\label{eq:CQ-consistency}
\frac{\delta(e^{-z\tau})}{\tau}
=
z+O\!\left(\tau^p|z|^{p+1}\right),
\qquad |z\tau|\to0.
\end{equation}
The first condition keeps the discrete Laplace frequencies inside the
analytic sector of $F$, while the second provides the local
approximation of the continuous Laplace variable.

A fundamental property of CQ is the exact discrete composition rule
\begin{equation}
\label{eq:CQ-composition}
F_2(\partial_\tau)F_1(\partial_\tau)g
=
(F_2F_1)(\partial_\tau)g,
\end{equation}
which follows directly from multiplication of the corresponding
generating functions. This is the discrete counterpart of the
associativity of continuous convolution and is central to the
stability of convolution equations.

The construction has three features that are particularly relevant
for memory problems. First, only the Laplace symbol $F$ is required,
so an explicit representation of the time-domain kernel is unnecessary. Second, singular kernels and kernels containing several time scales are treated through the same generating-function formula. Third, the exact operational identity \eqref{eq:CQ-composition} preserves the algebraic structure of convolution and integro-differential equations. These properties make CQ well suited to fractional and hereditary models, as well as to spatially discretized evolution problems whose transfer operators are naturally available in the Laplace domain.

\subsection{Regular variation}
\label{subsec:regular-variation}

We recall several standard facts about regularly varying functions
that will be used in the analysis below; see
\cite{seneta1976regularly,bingham1987regular}.

\begin{definition}
\label{def:regular-variation}
A measurable function
$\Phi:(0,\infty)\to(0,\infty)$ is said to be regularly varying at
infinity with index $\ell\in\mathbb{R}$ if
\[
\lim_{r\to\infty}
\frac{\Phi(\lambda r)}{\Phi(r)}
=
\lambda^\ell
\]
for every fixed $\lambda>0$. In this case, we write
$\Phi\in RV_\ell$. In particular, $\Phi\in RV_0$ is called slowly
varying. Equivalently, every $\Phi\in RV_\ell$ can be written as
\[
\Phi(r)=r^\ell L(r),
\]
where $L\in RV_0$.
\end{definition}

\begin{definition}\label{def:two-endpoint-regular-variation}
A measurable function $\Phi:(0,\infty)\to(0,\infty)$ is regularly
varying with indices $\rho$ at zero and $\ell$ at infinity, if, for every fixed $\lambda>0$,
\[
\lim_{r\to0}\frac{\Phi(\lambda r)}{\Phi(r)}=\lambda^\rho,
\qquad
\lim_{r\to\infty}\frac{\Phi(\lambda r)}{\Phi(r)}=\lambda^\ell.
\]
The indices $\rho,\ell\in\mathbb R$ may differ.
\end{definition}

Regular variation of $\Phi$ at zero with index
$\rho$ is equivalent to regular variation of
$\Psi(r)=\Phi(r^{-1})$ at infinity with index
$-\rho$, since
\[
\lim_{r\to\infty}
\frac{\Psi(\lambda r)}{\Psi(r)}
=
\lim_{s\to0^+}
\frac{\Phi(\lambda^{-1}s)}{\Phi(s)}
=\lambda^{-\rho}.
\]
Hence the Potter bound at infinity applied to
$\Psi$ yields the corresponding estimate for
$\Phi$ near zero.

\begin{lemma}[Potter bound \cite{bingham1987regular}]
\label{lem:potter}
Let $\Phi\in RV_\ell$. Then, for every $\varepsilon>0$ and
$A>1$, there exists $r_0>0$ such that
\[
\frac{\Phi(y)}{\Phi(x)}
\le
A
\max\left\{
\left(\frac{y}{x}\right)^{\ell+\varepsilon},
\left(\frac{y}{x}\right)^{\ell-\varepsilon}
\right\},
\qquad x,y\ge r_0.
\]
\end{lemma}

For estimates valid at all positive times, we need a comparison
on the whole positive axis.

\begin{lemma}\label{lem:potter-global}
Let $\Phi:(0,\infty)\to(0,\infty)$ be continuous
and regularly varying at zero and infinity with
indices $\rho$ and $\ell$, respectively.
For every $\varepsilon>0$, there exists a constant
$C>0$ such that
\begin{equation}\label{eq:global-potter}
\frac{\Phi(y)}{\Phi(x)}
\le C\max\left\{
\left(\frac yx\right)^{\min\{\rho,\ell\}-\varepsilon},
\left(\frac yx\right)^{\max\{\rho,\ell\}+\varepsilon}
\right\},
\qquad x,y>0.
\end{equation}
The constant $C$ is independent of $x$ and $y$.
\end{lemma}

\begin{proof}
Since $\Phi$ is regularly varying at zero with
index $\rho$, the function
$\Psi(r)=\Phi(r^{-1})$ is regularly varying
at infinity with index $-\rho$.
By Lemma~\ref{lem:potter}, for sufficiently
small $a>0$ and $0<x\le y\le a$,
\[
\begin{aligned}
\frac{\Phi(y)}{\Phi(x)}
&=\frac{\Psi(y^{-1})}{\Psi(x^{-1})}\\
&\le C\max\left\{
\left(\frac{x}{y}\right)^{-\rho-\varepsilon},
\left(\frac{x}{y}\right)^{-\rho+\varepsilon}
\right\}\\
&=C\left(\frac{y}{x}\right)^{\rho+\varepsilon}.
\end{aligned}
\]
Applying the same Potter bound to the
reciprocal ratio gives
\[
\frac{\Phi(x)}{\Phi(y)}
\le C\left(\frac{x}{y}\right)^{\rho-\varepsilon}.
\]
Consequently,
\[
C^{-1}\left(\frac yx\right)^{\rho-\varepsilon}
\le\frac{\Phi(y)}{\Phi(x)}
\le C\left(\frac yx\right)^{\rho+\varepsilon},
\qquad 0<x\le y\le a.
\]

Similarly, regular variation at infinity
with index $\ell$ and Lemma~\ref{lem:potter}
yield, for sufficiently large $b>a$,
\[
C^{-1}\left(\frac yx\right)^{\ell-\varepsilon}
\le\frac{\Phi(y)}{\Phi(x)}
\le C\left(\frac yx\right)^{\ell+\varepsilon},
\qquad b\le x\le y.
\]

Since $\Phi$ is positive and continuous on
$[a,b]$, both $\Phi(y)/\Phi(x)$ and its
reciprocal are uniformly bounded for
$a\le x\le y\le b$.
Moreover, $1\le y/x\le b/a$, so the
corresponding power bounds also hold
on this compact interval.
Combining these estimates, we obtain
\begin{equation}\label{eq:global-potter-ordered}
C^{-1}\left(\frac yx\right)^{\min\{\rho,\ell\}-\varepsilon}
\le\frac{\Phi(y)}{\Phi(x)}
\le C\left(\frac yx\right)^{\max\{\rho,\ell\}+\varepsilon}
\end{equation}
whenever $x\le y$ belong to the one of the three intervals
$(0,a]$, $[a,b]$, and $[b,\infty)$.

For $x\le y$ in different intervals, we split the ratio at the intermediate endpoints. Thus \eqref{eq:global-potter-ordered}
holds for all $0<x\le y$.

Finally, if $y<x$, applying the lower bound
in \eqref{eq:global-potter-ordered} to
$\Phi(x)/\Phi(y)$ and taking reciprocals yields
\[
\frac{\Phi(y)}{\Phi(x)}
\le C\left(\frac yx\right)^{\min\{\rho,\ell\}-\varepsilon}.
\]
Together with the upper bound for $x\le y$,
this proves \eqref{eq:global-potter}.
\end{proof}

The global estimate also implies a useful
scaling bound. Indeed, for any
$\nu>\max\{\rho,\ell,0\}$, choose
$\varepsilon>0$ sufficiently small such that
$\max\{\rho,\ell\}+\varepsilon<\nu$.
Then, by \eqref{eq:global-potter}, for
$r>0$ and $\eta\ge1$,
\[
\begin{aligned}
\frac{\Phi(\eta r)}{\Phi(r)}
&\le C\max\left\{
\eta^{\min\{\rho,\ell\}-\varepsilon},
\eta^{\max\{\rho,\ell\}+\varepsilon}
\right\}\\
&=C\eta^{\max\{\rho,\ell\}+\varepsilon}
\le C\eta^\nu.
\end{aligned}
\]
Consequently,
\begin{equation}\label{eq:RV-rescaling}
\Phi(\eta r)\le C\eta^\nu\Phi(r),
\qquad r>0,\quad \eta\ge1,
\end{equation}
where $C$ is independent of $r$ and $\eta$.

\section{Convolution quadrature estimates under regularly varying radial bounds}
\label{sec:regularly-varying-majorants}

We impose the following assumption on the radial bound $\Phi$.

\textbf{Assumption.}\label{ass:two-endpoint-radial}
The operator-valued function $F:\Sigma_{\pi-\theta}\to\mathcal L(X)$
is analytic and satisfies
\[
|F(z)|\le C\Phi(|z|),\qquad z\in\Sigma_{\pi-\theta},
\]
where $\Phi:(0,\infty)\to(0,\infty)$ is positive, continuous, and regularly varying at zero and infinity with indices $\rho,\ell\in\mathbb R$, respectively. Under this assumption, we obtain the following approximation and stability estimates.

\begin{theorem}\label{thm:CQ-kernel}
Let Assumption~1 hold, and assume that the underlying linear multistep method is strongly $A(\alpha)$-stable and has order $p$, with $\alpha>\theta$. Then, for every $\tau>0$ and $n\ge1$,
\begin{align}
\|f_n-f(t_n)\|
&\le C\tau^pt_n^{-p-1}\Phi(t_n^{-1}),
\label{eq:CQ-kernel-error}\\
\|f_n\|+\|f(t_n)\|
&\le Ct_n^{-1}\Phi(t_n^{-1}),
\label{eq:CQ-kernel-stability}
\end{align}
where $t_n=n\tau$ and the constant $C$ is independent of $\tau$ and $n$.
\end{theorem}

\begin{proof}

The proof extends the proof of \cite[Theorem~2.1]{lubich2004convolution}, while keeping the radial estimates explicit. Fix $\nu>\max\{\rho,\ell,0\}$ and write $\widetilde F(z)=z^{-k}F(z)$, where the integer $k$ is chosen such that $k>\nu+1$. For $\xi$ sufficiently close to zero, Cauchy's formula applied to $\widetilde F$ yields
\[
\begin{aligned}
\sum_{n=0}^{\infty}f_n\xi^n
&=\frac1\tau
F\left(\frac{\delta(\xi)}{\tau}\right)\\
&=\frac{1}{2\pi\mathrm i}
\int_\Gamma F(\lambda)
\left(\frac{\delta(\xi)}{\tau\lambda}\right)^k
\frac{1}{\delta(\xi)-\tau\lambda}
\,\mathrm d\lambda,
\end{aligned}
\]
where $\Gamma$ is a contour in the sector of analyticity, parallel to its boundary and oriented with increasing imaginary part. Define $e_n(z)$ by the generating function
\[
\sum_{n=0}^{\infty}e_n(z)\xi^n
=
\left(\frac{\delta(\xi)}{z}\right)^k
\frac{1}{\delta(\xi)-z}.
\]
Extracting the coefficients of $\xi^n$ gives
\[
f_n=\frac{1}{2\pi\mathrm i}
\int_\Gamma F(\lambda)e_n(\tau\lambda)
\,\mathrm d\lambda.
\]
On the other hand, the inverse Laplace transform gives, for $t_n>0$,
\[
f(t_n)=\frac{1}{2\pi\mathrm i}
\int_\Gamma e^{t_n\lambda}F(\lambda)
\,\mathrm d\lambda.
\]
Consequently, subtracting the two representations yields
\begin{equation}\label{eq:modified-error-representation}
f_n-f(t_n)
=\frac{1}{2\pi\mathrm i}
\int_\Gamma F(\lambda)
\left[e_n(\tau\lambda)-e^{t_n\lambda}\right]
\,\mathrm d\lambda.
\end{equation}

Fix $\theta<\alpha'<\min\{\alpha,\pi/2\}$. By \cite[eqs.~(3.5)--(3.6)]{lubich2004convolution}, there exist a sufficiently small $r>0$ and a constant $\varrho\in(0,1)$ such that, for $z=\tau\lambda$,
\begin{equation}\label{approx_exp1}
\begin{aligned}
|e_n(z)-e^{nz}|
&\le C\left(\left|z^pe^{nz/2}\right|
+\frac{\varrho^n}{|z|^k}\right),
&&\text{for }0<|z|\le r,\quad |\arg(-z)|\le\alpha',\\
|e_n(z)-e^{nz}|
&\le C\left(|z|^p e^{2n|z|}
+\frac{\varrho^n}{|z|^k}\right),
&&\text{for }0<|z|\le r,
\end{aligned}
\end{equation}
and
\begin{equation}\label{approx_exp2}
|e_n(z)|\le C\frac{\varrho^n}{|z|^{k+1}},
\qquad\text{for } |z|\ge r,\quad |\arg(-z)|\le\alpha'.
\end{equation}
For each $t_n=n\tau$, deform $\Gamma$ to the circular arc
\[
\Gamma^0=\left\{\frac r{t_n}e^{\mathrm i\varphi}:
-\pi+\alpha'\le\varphi\le\pi-\alpha'\right\}
\]
and the rays
\[
\begin{aligned}
\Gamma_{\mathrm{near}}^\pm
&=\left\{se^{\pm\mathrm i(\pi-\alpha')}:
\frac{r}{t_n}\le s\le \frac{r}{\tau}\right\},\\
\Gamma_{\mathrm{far}}^\pm
&=\left\{se^{\pm\mathrm i(\pi-\alpha')}:
s\ge \frac{r}{\tau}\right\}.
\end{aligned}
\]

On the arc, parameterize $\lambda=(r/t_n)e^{\mathrm i\varphi}$.
Since $|\tau\lambda|=r/n$ and $|t_n\lambda|=r$, the second
bound in \eqref{approx_exp1} yields
\begin{equation}\label{eq:arc}
\begin{aligned}
\left|\int_{\Gamma^0}F(\lambda)
[e_n(\tau\lambda)-e^{t_n\lambda}]\,\mathrm d\lambda\right|
&\le C\frac r{t_n}\Phi(r/t_n)
\int_{-\pi+\alpha'}^{\pi-\alpha'}
\left[\left(\frac rn\right)^pe^{2r}
+\varrho^n\left(\frac nr\right)^k\right]\,\mathrm d\varphi\\
&\le Ct_n^{-1}\Phi(t_n^{-1})
\left(n^{-p}+\varrho^nn^k\right)\\
&\le C\tau^pt_n^{-p-1}\Phi(t_n^{-1}).
\end{aligned}
\end{equation}
The last step follows from
$\sup_{n\ge1}\varrho^nn^{k+p}<\infty$ and $n^{-p}=\tau^pt_n^{-p}$.

On each near ray, the first bound in \eqref{approx_exp1} gives
\begin{equation}\label{eq:near-split}
\begin{aligned}
\left|\int_{\Gamma_{\mathrm{near}}^\pm}F(\lambda)
[e_n(\tau\lambda)-e^{t_n\lambda}]\,\mathrm d\lambda\right|
&\le C\int_{r/t_n}^{r/\tau}\Phi(s)
\left[(\tau s)^pe^{-ct_ns/2}
+\frac{\varrho^n}{(\tau s)^k}\right]\,\mathrm ds\\
&\le C\Phi(t_n^{-1})
\int_{r/t_n}^{r/\tau}(t_ns)^\nu
\left[(\tau s)^pe^{-ct_ns/2}
+\frac{\varrho^n}{(\tau s)^k}\right]\,\mathrm ds\\
&= Ct_n^{-1}\Phi(t_n^{-1})
\left[n^{-p}\int_r^{rn}y^{\nu+p}e^{-cy/2}\,\mathrm dy
+\varrho^nn^k\int_r^{rn}y^{\nu-k}\,\mathrm dy\right],
\end{aligned}
\end{equation}
where the second inequality follows from \eqref{eq:RV-rescaling} with $\eta=t_ns/r\ge1$. Setting $y=t_ns$, the first integral in the last equation satisfies
\[
\int_r^\infty y^{\nu+p}e^{-cy/2}\,\mathrm dy
<\infty.
\]
Since $k>\nu+1$, the second integral satisfies
\[
\int_r^{rn}y^{\nu-k}\,\mathrm dy
\le\int_r^\infty y^{\nu-k}\,\mathrm dy
<\infty.
\]
Moreover, since $0<\varrho<1$, $\varrho^n n^{k+p}$ is uniformly bounded for $n\ge1$. Consequently, \eqref{eq:near-split} is bounded by
\[
Ct_n^{-1}\Phi(t_n^{-1})(n^{-p}+\varrho^nn^k)
\le C\tau^pt_n^{-p-1}\Phi(t_n^{-1}).
\]

On each far ray, \eqref{approx_exp2} and the triangle inequality yield
\begin{equation}\label{eq:far-split}
\begin{aligned}
\left|\int_{\Gamma_{\mathrm{far}}^\pm}F(\lambda)
[e_n(\tau\lambda)-e^{t_n\lambda}]\,\mathrm d\lambda\right|
&\le C\int_{r/\tau}^\infty\Phi(s)
\left[\frac{\varrho^n}{(\tau s)^{k+1}}+e^{-ct_ns}\right]\,\mathrm ds\\
&\le C\Phi(t_n^{-1})
\int_{r/\tau}^\infty(t_ns)^\nu
\left[\frac{\varrho^n}{(\tau s)^{k+1}}+e^{-ct_ns}\right]\,\mathrm ds\\
&= Ct_n^{-1}\Phi(t_n^{-1})
\left[\varrho^nn^{k+1}\int_{rn}^\infty y^{\nu-k-1}\,\mathrm dy
+\int_{rn}^\infty y^\nu e^{-cy}\,\mathrm dy\right].
\end{aligned}
\end{equation}
The two integrals satisfy
\begin{equation*}
\begin{aligned}
\int_{rn}^\infty y^{\nu-k-1}\,\mathrm dy
&=\frac{(rn)^{\nu-k}}{k-\nu}\\
\int_{rn}^\infty y^\nu e^{-cy}\,\mathrm dy
&\le (rn)^{-p}\int_r^\infty y^{\nu+p}e^{-cy}\,\mathrm dy\\
&\le Cn^{-p}.
\end{aligned}
\end{equation*}
Thus \eqref{eq:far-split} is bounded by
\[
Ct_n^{-1}\Phi(t_n^{-1})
\left(\varrho^nn^{\nu+1}+n^{-p}\right)
\le C\tau^pt_n^{-p-1}\Phi(t_n^{-1}),
\]
because $\sup_{n\ge1}\varrho^nn^{\nu+p+1}<\infty$.
Adding the arc and the four ray contributions in
\eqref{eq:modified-error-representation} proves
\eqref{eq:CQ-kernel-error}.

To establish the stability estimate, we apply the same contour to the inverse Laplace to obtain
\[
\begin{aligned}
\|f(t_n)\|
&\le C\frac r{t_n}\Phi(r/t_n)e^r
+C\int_{r/t_n}^\infty\Phi(s)e^{-ct_ns}\,\mathrm ds\\
&\le Ct_n^{-1}\Phi(t_n^{-1})
\left(1+\int_r^\infty y^\nu e^{-cy}\,\mathrm dy\right)\\
&\le Ct_n^{-1}\Phi(t_n^{-1}).
\end{aligned}
\]
Hence
\[
\|f_n\|\le\|f_n-f(t_n)\|+\|f(t_n)\|
\le Ct_n^{-1}\Phi(t_n^{-1}),
\]
which proves \eqref{eq:CQ-kernel-stability}. All comparison
constants and all integrals after rescaling are independent of
$t_n$, so the estimates hold uniformly for all positive grid times.
\end{proof}

\begin{corollary}\label{cor:CQ-modified-symbol}
Under the assumptions of Theorem~\ref{thm:CQ-kernel}, let
$m\in\mathbb R$ and define
\[
F^{(m)}(z)=z^mF(z),\qquad f^{(m)}=\mathcal L^{-1}\{F^{(m)}\}.
\]
Define the corresponding weights and scaled weights by
\[
\sum_{n=0}^\infty\omega_n^{(m)}\xi^n
=F^{(m)}\!\left(\frac{\delta(\xi)}\tau\right),
\qquad f_n^{(m)}=\frac{\omega_n^{(m)}}\tau.
\]
Then, for all $t_n=n\tau>0$,
\begin{align*}
\|f_n^{(m)}-f^{(m)}(t_n)\|
&\le C\tau^pt_n^{-m-p-1}\Phi(t_n^{-1}),\\
\|f_n^{(m)}\|+\|f^{(m)}(t_n)\|
&\le Ct_n^{-m-1}\Phi(t_n^{-1}).
\end{align*}
The constant $C$ is independent of $\tau$, $n$ and $t_n$.
\end{corollary}

\begin{proof}
Since $\|F^{(m)}(z)\|\le C|z|^m\Phi(|z|)$, and $r^m\Phi(r)$ is positive, continuous, and regularly varying at zero and infinity with indices $m+\rho$ and $m+\ell$, respectively, the desired result follows by applying Theorem~\ref{thm:CQ-kernel} to $F^{(m)}$.
\end{proof}

\begin{theorem}\label{thm:CQ-power-data}
Under the assumptions of Theorem~\ref{thm:CQ-kernel}, let
$g_\beta(t)=t^{\beta-1}/\Gamma(\beta)$ for $t>0$, where $\beta>0$. Set $g_\beta(0)=0$ for $0<\beta<1$. Then, for all $t_n=n\tau>0$,
\begin{equation}\label{eq:CQ-power-data-error}
\|F(\partial_\tau)g_\beta(t_n)-F(\partial_t)g_\beta(t_n)\|
\le
\begin{cases}
C\tau^\beta t_n^{-1}\Phi(t_n^{-1}),&0<\beta\le p,\\[1mm]
C\tau^p t_n^{\beta-p-1}\Phi(t_n^{-1}),&\beta\ge p.
\end{cases}
\end{equation}
The constant is independent of $\tau$, $n$ and $t_n$.
\end{theorem}

\begin{proof}
Fix $\nu>\max\{\rho,\ell,0\}$ and an integer
$N>\max\{p,\beta+\nu\}$. For $k\in\mathbb{R}$, let
$\delta_j^{(k)}$ be defined by $\sum_{j=0}^{\infty}\delta_j^{(k)}\xi^j=\delta(\xi)^k$. By \cite[Lemma~4.1]{lubich2004convolution}, there exist constants
$c_p,\ldots,c_{N-1}$ such that
\begin{equation}
\label{eq:power-sequence-expansion}
\frac{j^{\beta-1}}{\Gamma(\beta)}
=
\delta_j^{(-\beta)}
+
\sum_{l=p}^{N-1}
c_l\delta_j^{(l-\beta)}
+
r_j,
\quad j\geq1,
\end{equation}
where $|r_j|\leq cj^{\beta-N-1}$. At $j=0$, interpret $j^{\beta-1}/\Gamma(\beta)$ as $1$ when $\beta=1$ and as $0$ otherwise. Define $r_0$ by \eqref{eq:power-sequence-expansion} using this convention. Furthermore,
\begin{equation}
\label{eq:power-sequence-remainder-unified}
|r_j|
\leq
c(j+1)^{\beta-N-1},
\quad j\geq0.
\end{equation}

Since $\omega_j=\tau f_j$, the discrete convolution is
\[
F(\partial_\tau)g_\beta(t_n)
=\tau^\beta\sum_{j=0}^n f_{n-j}
\frac{j^{\beta-1}}{\Gamma(\beta)}.
\]
Using \eqref{eq:power-sequence-expansion}, we obtain
\begin{equation}\label{eq:discrete-power-expansion}
F(\partial_\tau)g_\beta(t_n)
=f_n^{(-\beta)}
+\sum_{l=p}^{N-1}c_l\tau^lf_n^{(l-\beta)}+\mathcal R_n,
\end{equation}
where $\mathcal R_n=\tau^\beta\sum_{j=0}^n f_{n-j}r_j$.

We first estimate the remainder $\mathcal R_n$. Theorem~\ref{thm:CQ-kernel} gives $\|f_l\|\le C(l\tau)^{-1}\Phi((l\tau)^{-1})$ for $l\ge1$. For $l=0$, the definition of the convolution quadrature weights gives
\[
\|f_0\|=\tau^{-1}\left\|F\!\left(\frac{\delta(0)}\tau\right)\right\|
\le C\tau^{-1}\Phi(|\delta(0)|/\tau)
\le C\tau^{-1}\Phi(\tau^{-1}).
\]
Thus, for every $l\ge0$,
\begin{equation}\label{eq:unified-fq-bound}
\|f_l\|\le C\tau^{-1}(l+1)^{-1}
\Phi\!\left(((l+1)\tau)^{-1}\right).
\end{equation}
For $0\le j\le n$, the ratio $n/(n-j+1)$ is at least $1/2$.
Equation~\eqref{eq:RV-rescaling} therefore gives
\begin{equation}\label{eq:Phi-grid-comparison}
\Phi\!\left(((n-j+1)\tau)^{-1}\right)
\le C\Phi(t_n^{-1})
\left(\frac{n}{n-j+1}\right)^\nu.
\end{equation}
Combining \eqref{eq:power-sequence-remainder-unified},
\eqref{eq:unified-fq-bound}, and \eqref{eq:Phi-grid-comparison},
we obtain
\[
\|\mathcal R_n\|
\le C\tau^{\beta-1}\Phi(t_n^{-1})n^\nu
\sum_{j=0}^n(n-j+1)^{-\nu-1}(j+1)^{\beta-N-1}.
\]
For $0\le j\le\lfloor n/2\rfloor$, $n-j+1\ge n/2$ and
$N>\beta$, so
\[
\sum_{j=0}^{\lfloor n/2\rfloor}
(n-j+1)^{-\nu-1}(j+1)^{\beta-N-1}
\le Cn^{-\nu-1}\sum_{j=0}^\infty(j+1)^{\beta-N-1}
\le Cn^{-\nu-1}.
\]
For $\lfloor n/2\rfloor<j\le n$, $j+1\ge n/2$ and $\nu>0$,
so the remaining sum satisfies
\[
\begin{aligned}
\sum_{j=\lfloor n/2\rfloor+1}^n
(n-j+1)^{-\nu-1}(j+1)^{\beta-N-1}
&\le Cn^{\beta-N-1}\sum_{l=0}^\infty(l+1)^{-\nu-1}\\
&\le Cn^{\beta-N-1}\le Cn^{-\nu-1},
\end{aligned}
\]
where the last inequality uses $N>\beta+\nu$. It follows that
\begin{equation}\label{eq:remainder-estimate-preliminary}
\|\mathcal R_n\|
\le C\tau^{\beta-1}n^{-1}\Phi(t_n^{-1})
=C\tau^\beta t_n^{-1}\Phi(t_n^{-1}).
\end{equation}

Since $\mathcal L\{g_\beta\}(z)=z^{-\beta}$, the continuous
response is $F(\partial_t)g_\beta(t_n)=f^{(-\beta)}(t_n)$.
Subtracting it from \eqref{eq:discrete-power-expansion} gives
\begin{equation}\label{eq:power-data-error-decomposition}
\begin{aligned}
F(\partial_\tau)g_\beta(t_n)-F(\partial_t)g_\beta(t_n)
=f_n^{(-\beta)}-f^{(-\beta)}(t_n)
+\sum_{l=p}^{N-1}c_l\tau^lf_n^{(l-\beta)}+\mathcal R_n.
\end{aligned}
\end{equation}
Corollary~\ref{cor:CQ-modified-symbol}, first with $m=-\beta$
and then with $m=l-\beta$, gives
\begin{align}
\|f_n^{(-\beta)}-f^{(-\beta)}(t_n)\|
&\le C\tau^pt_n^{\beta-p-1}\Phi(t_n^{-1}),
\label{eq:leading-power-error}\\
\|\tau^lf_n^{(l-\beta)}\|
&\le C\tau^lt_n^{\beta-l-1}\Phi(t_n^{-1}).
\label{eq:correction-power-bound}
\end{align}
If $0<\beta\le p$, then, for $l\ge p$,
\[
\begin{aligned}
\tau^pt_n^{\beta-p-1}
&=\tau^\beta t_n^{-1}(\tau/t_n)^{p-\beta}
\le\tau^\beta t_n^{-1},\\
\tau^lt_n^{\beta-l-1}
&=\tau^\beta t_n^{-1}(\tau/t_n)^{l-\beta}
\le\tau^\beta t_n^{-1}.
\end{aligned}
\]
If $\beta\ge p$, then
\[
\begin{aligned}
\tau^lt_n^{\beta-l-1}
&=\tau^pt_n^{\beta-p-1}(\tau/t_n)^{l-p}
\le\tau^pt_n^{\beta-p-1},\\
\tau^\beta t_n^{-1}
&=\tau^pt_n^{\beta-p-1}(\tau/t_n)^{\beta-p}
\le\tau^pt_n^{\beta-p-1}.
\end{aligned}
\]
Using these inequalities together with
\eqref{eq:remainder-estimate-preliminary},
\eqref{eq:leading-power-error}, and
\eqref{eq:correction-power-bound} in
\eqref{eq:power-data-error-decomposition}, we obtain
\eqref{eq:CQ-power-data-error} with constants uniform for $t_n>0$.
\end{proof}

\begin{remark}\label{rem:two-endpoint-time-scales}
The radial assumption admits symbols that satisfy neither a
nonincreasing radial bound nor the classical global single-power
bound on the unshifted sector. Consider the example
\[
F(z)=z+z^{-1},\qquad \Phi(r)=r+r^{-1},\qquad r>0.
\]
The symbol is analytic in $\Sigma_{\pi-\theta}$ and satisfies
$|F(z)|\le\Phi(|z|)$. The function $\Phi$ is positive and
continuous, and, for every $\lambda>0$,
\[
\lim_{r\to0}\frac{\Phi(\lambda r)}{\Phi(r)}=\lambda^{-1},
\qquad
\lim_{r\to\infty}\frac{\Phi(\lambda r)}{\Phi(r)}=\lambda.
\]
Thus the radial assumption holds with $\rho=-1$ and $\ell=1$. The inverse Laplace transform of $F$ is $\delta_0'+1$, where $\delta_0'$ denotes the derivative of the Dirac distribution. Moreover, $F$ does not admit a global single-power bound on the unshifted sector.
\end{remark}

\section{Applications to fractional viscoelasticity and memory operators} \label{sec:applications}
The estimates in Section~\ref{sec:regularly-varying-majorants} give the CQ error in terms of a radial bound for the Laplace symbol. We apply them to constitutive responses in fractional viscoelasticity, distributed-order, and logarithmic-power memory models. Throughout this section, the multistep method satisfies the stability and order conditions of the theorem. The operators $I_t^q$ and $D_t^q$ are the Riemann--Liouville integral and derivative, with Laplace $z^{-q}$ and $z^q$. Here, complex powers and logarithms use the principal branch.

Let $g_\beta(t)=t^{\beta-1}/\Gamma(\beta)$ with $\beta>0$ and set
$u(t)=F(\partial_t)g_\beta(t)$. Its CQ approximation is defined by
\[
F\left(\frac{\delta(\xi)}{\tau}\right)
 =\sum_{j=0}^{\infty}\omega_j\xi^j,
\qquad
u_\tau^n=\sum_{j=0}^n\omega_jg_\beta(t_{n-j}).
\]
Theorem~\ref{thm:CQ-power-data} yields
\begin{equation}\label{eq:application-power-error}
|u_\tau^n-u(t_n)|
\leq C\tau^{\min\{p,\beta\}}
 t_n^{\beta-\min\{p,\beta\}-1}\Phi(t_n^{-1}).
\end{equation}
CQ error estimates for equations with memory are developed in
\cite{lubich1988convolution,lubich2004convolution,
lubich1996nonsmooth,cuesta2006convolution}. A particular hypothesis
is the global bound $|F(z)|\leq M|z|^{-\mu}$ on an unshifted
sector, as in \cite[eq.~(4.1)]{bazhlekova2015analysis}.
The framework of \cite[eq.~(2.1)]{lubich2004convolution} also
allows shifted sectors. Consequently, failure of a global power
bound on the unshifted sector does not by itself exclude a
classical CQ analysis. A finite sum of power symbols can also be
estimated term by term. The estimates used here apply directly to the complete symbol and retain the radial factor $\Phi(t_n^{-1})$.

\subsection{Fractional viscoelastic models}
\label{subsec:viscoelastic-applications}

Fractional Maxwell and Kelvin--Voigt models are obtained by connecting fractional viscoelastic elements in series and in parallel, respectively. These models describe power-type creep and relaxation in complex fluids and soft materials; see
\cite{jaishankar2013power,bonfanti2020fractional}. Here the quantities of interest are strain under a prescribed stress and stress under a prescribed strain.

\subsubsection{Fractional Maxwell model}

Let \(0<a<b<1\) and \(c_1,c_2>0\). The constitutive equation is
\[
D_t^b\varepsilon
=c_1D_t^{b-a}\sigma+c_2\sigma,
\]
where \(\sigma\) and \(\varepsilon\) denote stress and strain.
Under zero history, this is equivalent to
\[
\varepsilon=c_1I_t^a\sigma+c_2I_t^b\sigma,
\qquad
F(z)=c_1z^{-a}+c_2z^{-b}.
\]
The triangle inequality gives $|F(z)|\leq\Phi(|z|)$ with
$\Phi(r)=c_1r^{-a}+c_2r^{-b}$. This positive continuous function is regularly varying at zero with index $-b$ and at infinity with index $-a$, and hence satisfies the radial assumption. A global power bound would require $\mu\geq b$ as $r\to0$ and $\mu\leq a$ as $r\to\infty$, which are incompatible.

For the power-type stress
\(\sigma(t)=g_\beta(t)=t^{\beta-1}/\Gamma(\beta)\), the exact
strain is
\[
\varepsilon(t)
=\frac{c_1t^{\beta+a-1}}{\Gamma(\beta+a)}
+\frac{c_2t^{\beta+b-1}}{\Gamma(\beta+b)}.
\]
Writing \(\varepsilon_\tau^n\) for its convolution quadrature
approximation, \eqref{eq:application-power-error} gives
\begin{equation}\label{eq:application-maxwell-error}
|\varepsilon_\tau^n-\varepsilon(t_n)|
\leq C\tau^{\min\{p,\beta\}}
t_n^{\beta-\min\{p,\beta\}-1}
\bigl(c_1t_n^a+c_2t_n^b\bigr).
\end{equation}
For a unit step stress ($\beta=1$), \eqref{eq:application-maxwell-error} reduces to
\[
C\tau(c_1t_n^{a-1}+c_2t_n^{b-1}).
\]
For fixed orders and a fixed multistep method, the constant may be chosen independently of $c_1,c_2$. This follows from the corresponding estimates for the two power symbols $z^{-a}$ and $z^{-b}$. Thus the displayed bound records the two material contributions explicitly.

\subsubsection{Fractional Kelvin--Voigt model}

For the parallel constitutive law
\begin{equation}\label{eq:application-kelvin-voigt-law}
\sigma=c_1D_t^a\varepsilon+c_2D_t^b\varepsilon,
\qquad 0<a<b<1,
\end{equation}
the strain-to-stress symbol and its radial bound are
\[
F(z)=c_1z^a+c_2z^b,
\qquad \Phi(r)=c_1r^a+c_2r^b.
\]
For $c_1,c_2>0$, the positive continuous function $\Phi$ satisfies the radial assumption, with index $a$ at zero and index $b$ at infinity. Indeed,
\[
\Phi(r)\sim c_1r^a \text{ as } r\to0,
\qquad
\Phi(r)\sim c_2r^b \text{ as } r\to\infty ,
\]
so no global single-power bound can describe both endpoint regimes.

For the ramp strain $\varepsilon(t)=t=g_2(t)$, the exact response is
\[
\sigma(t)=\frac{c_1t^{1-a}}{\Gamma(2-a)}
+\frac{c_2t^{1-b}}{\Gamma(2-b)}.
\]
Applying \eqref{eq:application-power-error} to the second-order backward differentiation formula, with $\delta(\xi)=3/2-2\xi+\xi^2/2$, yields
\begin{equation}\label{eq:application-kelvin-voigt-error}
|\sigma_\tau^n-\sigma(t_n)|
\leq C\tau^2(c_1t_n^{-a-1}+c_2t_n^{-b-1}).
\end{equation}
The estimate explicitly retains the contributions of both fractional orders and their coefficients to the error bound.

\subsection{Distributed-order model}
\label{subsec:distributed-memory}

Distributed-order models and their CQ approximations have been studied in \cite{jin2016error, yin2021approximation}. We consider a uniform distribution of fractional orders over $[0,1]$:
\begin{equation}\label{eq:application-distributed-memory}
\varepsilon(t)=\int_0^1 I_t^q\sigma(t)\,\mathrm dq.
\end{equation}
Its Laplace symbol and radial bound are
\[
F(z)=\int_0^1 z^{-q}\,\mathrm dq=\frac{1-z^{-1}}{\log z},
\qquad
\Phi(r)=\int_0^1 r^{-q}\,\mathrm dq=\frac{1-r^{-1}}{\log r},
\]
with $F(1)=\Phi(1)=1$. The function $\Phi$ is positive and continuous, with
\[
\Phi(r)\sim\frac{r^{-1}}{\log(1/r)}\quad(r\to0),
\qquad
\Phi(r)\sim\frac1{\log r}\quad(r\to\infty).
\]
Thus it satisfies the radial assumption with indices $\rho=-1$ and $\ell=0$. However, no global single-power bound holds for $F$ on the unshifted sector, because its behaviors near the origin and at infinity cannot be described by the same power of $|z|$.

For $\sigma(t)=g_\beta(t)=t^{\beta-1}/\Gamma(\beta)$, the exact
response and its CQ error bound are
\[
\begin{aligned}
\varepsilon(t)&=\int_0^1
\frac{t^{\beta+q-1}}{\Gamma(\beta+q)}\,\mathrm dq,\\
|\varepsilon_\tau^n-\varepsilon(t_n)|
&\le C\tau^{\min\{p,\beta\}}
t_n^{\beta-\min\{p,\beta\}-1}
\frac{1-t_n}{\log(1/t_n)}.
\end{aligned}
\]
Here \eqref{eq:application-power-error} is applied to the complete symbol, and the quotient is assigned its continuous value $1$ at $t_n=1$. By treating the distributed-order symbol as a whole, the estimate preserves the logarithmic factor and avoids separate CQ estimates for each $q\in[0,1]$.

\subsection{Logarithmic-power memory model}
\label{subsec:logarithmic-power-memory}

Power-type kernels are a standard representation of long memory effects in fractional models. When the fractional order itself is an unknown parameter, its identification requires the sensitivity of the model response with respect to that order, which naturally leads to derivatives of fractional operators with respect to their order.
Differentiating a fractional integral with respect to its order produces a logarithmic correction to the underlying power-type kernel.
Such logarithmic-power operators appear, for example, in fractional
sensitivity equations for parameter identification
\cite{kharazmi2019fractional}, and weakly singular Volterra equations
with combined logarithmic-power kernels \cite{singh2024weakly}.
These considerations motivate the following derivative of a fractional integral with respect to its order.

For $0<a<1$, consider
\begin{equation*}
\mathcal{K}_a g(t)
:=
\frac{\partial}{\partial a}I_t^a g(t)
=
\int_0^t
\frac{(t-s)^{a-1}}{\Gamma(a)}
\bigl(\log(t-s)-\psi(a)\bigr)g(s)\,ds,
\end{equation*}
where $\psi(a)=\Gamma'(a)/\Gamma(a)$. Since
\[
\mathcal{L}\{I_t^a g\}(z)=z^{-a}\widehat g(z),
\]
the corresponding Laplace symbol is
\begin{equation*}
F(z)=-z^{-a}\log z.
\end{equation*}
Since $|F(r)|=r^{-a}|\log r|$ for $r>0$, no bound of the form $|F(z)|\le M|z|^{-\mu}$ can hold globally on $\Sigma_{\pi-\theta}$ for any fixed $\mu$:
the behavior as $r\to0$ requires $\mu>a$, whereas that as
$r\to\infty$ requires $\mu<a$. We now verify the radial assumption. Define
\[
\Phi(r)
=
r^{-a}\bigl(c_a+|\log r|\bigr),
\qquad
c_a:=1+\frac1a.
\]
Then $|F(z)|\le C\Phi(|z|)$. Moreover, $\Phi$ is positive and continuous and regularly varying with index $-a$ at both endpoints. Thus the radial assumption holds, and \eqref{eq:application-power-error} gives
\begin{equation*}
|u_\tau^n-u(t_n)|
\le
C\tau^{\min\{p,\beta\}}
t_n^{a+\beta-\min\{p,\beta\}-1}
\left(c_a+|\log t_n|\right).
\end{equation*}
The estimate therefore retains the logarithmic correction associated with the fractional-order sensitivity. Using instead the power estimate
\[
|\log t|\le C_\epsilon t^{-\epsilon},
\qquad 0<t\le1,
\]
would replace the logarithmically corrected factor
$t^a|\log t|$ by the coarser power $t^{a-\epsilon}$.

Section~\ref{sec:numerical-experiments} focuses on the
distributed-order operator introduced in
Section~\ref{subsec:distributed-memory}.
We first examine its response to the ramp input
$g_2(t)=t$, which admits an explicit integral
representation and allows us to verify the
convergence rates predicted by
\eqref{eq:application-power-error} for BE and SBD.
We then investigate the temporal convergence of
the corresponding PDE discretizations for smooth
and nonsmooth initial data.

Table~\ref{tab:radial-examples} summarizes representative Laplace symbols, admissible radial bounds, and their regular variation indices. Here $0<a<b<1$, $c_1,c_2>0$, and $c_a=1+a^{-1}$ in the logarithmic-power case. The indices $\rho$ and $\ell$ characterize the behavior of $\Phi(r)$ as $r\to0$ and $r\to\infty$, respectively. The table illustrates how the full factor $\Phi(t_n^{-1})$ retains distinct power-ltype scales and logarithmic corrections in the CQ error bounds.

\begin{table}[htbp]
\centering
\caption{Examples of radial symbol bounds, their regular
variation indices, and physical backgrounds.}
\label{tab:radial-examples}
\small
\setlength{\tabcolsep}{3pt}
\renewcommand{\arraystretch}{1.4}
\begin{tabularx}{\textwidth}{p{0.2\textwidth}
p{0.14\textwidth}p{0.14\textwidth}
p{0.08\textwidth}p{0.08\textwidth}X}
\toprule
Model & $F(z)$ & $\Phi(r)$ & $\rho$ & $\ell$
& Physical background \\
\midrule
Fractional integral
  & $z^{-a}$
  & $r^{-a}$
  & $-a$ & $-a$
  & Power-type memory \\
Fractional Maxwell
  & $c_1z^{-a}+c_2z^{-b}$
  & $c_1r^{-a}+c_2r^{-b}$
  & $-b$ & $-a$
  & Series connection of fractional viscoelastic elements \\
Fractional Kelvin--Voigt
  & $c_1z^a+c_2z^b$
  & $c_1r^a+c_2r^b$
  & $a$ & $b$
  & Parallel connection of fractional viscoelastic elements \\
Distributed order
  & $\displaystyle\int_0^1z^{-q}\,\mathrm dq$
  & $\displaystyle\int_0^1r^{-q}\,\mathrm dq$
  & $-1$ & $0$
  & Continuous distribution of fractional memory orders \\
Logarithmic power
  & $-z^{-a}\log z$
  & $r^{-a}(c_a+|\log r|)$
  & $-a$ & $-a$
  & Fractional-order sensitivity of power-type memory \\
\bottomrule
\end{tabularx}
\end{table}

\section{Numerical experiments}
\label{sec:numerical-experiments}
We investigate the distributed-order operator introduced in Section~\ref{subsec:distributed-memory}. First, its response to the ramp input $g_2(t)=t$ provides an explicit integral reference for assessing the convergence rates predicted by \eqref{eq:application-power-error}. We then study a corresponding evolution problem with smooth and nonsmooth initial data. We use backward Euler (BE) and second-order backward differentiation (SBD) convolution quadrature, generated by
\[
\delta(\xi)=1-\xi,
\qquad
\delta(\xi)=\frac32-2\xi+\frac12\xi^2,
\]
respectively, with the appropriate starting correction
for SBD. For convergence tests at the fixed final time
$T=0.1$, we take $N=5,10,20,40,80$ and $\tau=T/N$.
The short-time tests use SBD with $N=1000$ and
$\tau=t/N$ at $t=10^{-3},\ldots,10^{-8}$.

\subsection{Analysis of the distributed-order PDE}
\label{subsec:experiment-distributed}

Let $A=-\partial_{xx}$ denote the Dirichlet Laplacian
on $\Omega=(0,1)$, with
$D(A)=H^2(\Omega)\cap H_0^1(\Omega)$.
The operator $A$ is positive definite and self-adjoint
on $L^2(\Omega)$, and $-A$ generates a bounded analytic
semigroup. In particular, for any
$\theta\in(0,\pi/2)$,
\[
\|(zI+A)^{-1}\|
\leq C_\theta |z|^{-1},
\qquad z\in\Sigma_{\pi-\theta}.
\]

We consider the distributed-order evolution problem
\begin{equation}\label{eq:experiment-distributed-pde}
\begin{cases}
\displaystyle
\partial_t u+\int_0^1D_t^q Au\,\mathrm dq=0,
&x\in(0,1),\quad 0<t\le T,\\
u(0,t)=u(1,t)=0,
&0<t\le T,\\
u(x,0)=v(x),
&x\in(0,1),
\end{cases}
\end{equation}
where $D_t^q$ denotes the Riemann--Liouville derivative and the order density is uniform on $[0,1]$. Taking the Laplace transform gives
\[
z\widehat u(z)-v+
\int_0^1z^q A\widehat u(z)\,\mathrm dq=0.
\]
Hence,
\[
\widehat u(z)
=\frac1z(I+F(z)A)^{-1}v,
\]
where the integrated memory symbol is
\[
F(z)
=\int_0^1z^{q-1}\,\mathrm dq
=\int_0^1z^{-q}\,\mathrm dq
=\frac{1-z^{-1}}{\log z}.
\]
The inverse Laplace transform yields
\[
u(t)=\frac{1}{2\pi\mathrm i}
\int_{\Gamma_{\delta,\pi-\theta}}
e^{zt}\frac1z(I+F(z)A)^{-1}v\,\mathrm dz,
\]
where $\Gamma_{\delta,\pi-\theta}$ denotes the
standard sectorial contour.

We first establish sectorial bounds for the distributed-order
Laplace symbol $F(z)$. Let $g(z)=F(z)^{-1}$. For $z=re^{\mathrm i\varphi}\in\Sigma_{\pi-\theta}$, the arguments of $z^{-q}$ lie in an interval of width $|\varphi|$. Since the order integral has positive weights, we obtain $g(z)\in\Sigma_{\pi-\theta}$. Moreover, the triangle inequality and a rotation by $\varphi/2$ give
\[
\sin(\theta/2)\Phi(|z|)
\leq |F(z)|\leq\Phi(|z|),
\]
where
$\Phi(r)=\int_0^1r^{-q}\,\mathrm dq$.
Consequently,
\[
\frac1{\Phi(|z|)}
\leq |g(z)|
\leq C\frac1{\Phi(|z|)}.
\]

For the spatial discretization, we employ
continuous piecewise linear finite elements
on a uniform mesh with $h=2^{-11}$.
Let $A_h$ denote the discrete Dirichlet Laplacian,
defined by
\[
(A_hw_h,\chi_h)_{L^2(\Omega)}
=(\partial_xw_h,\partial_x\chi_h)_{L^2(\Omega)}
\]
for all finite element functions
$w_h,\chi_h\in V_h\subset H_0^1(\Omega)$.
We consider two initial data:
\begin{enumerate}
\item[(a)] smooth data: $v(x)=\sin(2\pi x)$;
\item[(b)] nonsmooth data: $v(x)=\chi_{(0,1/2]}(x)$.
\end{enumerate}
The initial approximation is $v_h=R_hv$ for the smooth initial data (a) and $v_h=P_hv$ for the nonsmooth initial data (b), where $R_h$
and $P_h$ are the Ritz and $L^2$ projections, respectively.

We consider the semidiscrete Galerkin approximation
of \eqref{eq:experiment-distributed-pde}:
find $u_h(t)\in V_h$ such that
\[
\partial_t u_h+\int_0^1D_t^q A_hu_h\,\mathrm dq=0,
\]
with $u_h(0)=v_h$. By integrating the semidiscrete equation from
$0$ to $t$, we arrive at
\[
u_h+F(\partial_t)A_hu_h=v_h.
\]
Applying the Laplace transform gives
\[
\widehat u_h(z)=z^{-1}G_h(z)v_h,
\qquad
G_h(z)=(I+F(z)A_h)^{-1}.
\]
By the estimate of $F(z)$, we have
\[
|1+\lambda F(z)|
\geq C(1+\lambda|F(z)|),
\qquad \lambda\geq0.
\]
Since $A_h$ is positive definite and
self-adjoint, it follows that
\begin{equation}\label{eq:distributed-resolvent-bounds}
\|G_h(z)\|+\|I-G_h(z)\|\leq C,
\qquad
\|F(z)G_h(z)\|\leq C\Phi(|z|),
\end{equation}
where $C$ is independent of $h$.

We now turn to the temporal discretization.
Upon introducing
\[
Q_h(z)=z(G_h(z)-I)
=-zF(z)G_h(z)A_h,
\]
the semidiscrete solution can be represented as
\[
u_h(t)=v_h+[Q_h(\partial_t)g_2](t)v_h,
\qquad g_2(t)=t.
\]
Following the convolution quadrature approach
in \cite{lubich1996nonsmooth,bazhlekova2015analysis},
we write the BE and corrected SBD approximations as
\begin{equation}\label{eq:distributed-increment-cq}
U_h^n
=v_h+[Q_h(\partial_\tau)g_2](t_n)v_h,
\end{equation}
and obtain the following time-discretization error bounds
\begin{equation}\label{eq:distributed-pde-time-error}
\|U_h^n-u_h(t_n)\|_{L^2(\Omega)}
\leq C\tau^p t_n^{-p}
\begin{cases}
\Phi(t_n^{-1})\|Av\|_{L^2(\Omega)},
& v\in D(A),\quad v_h=R_hv,\\
\|v\|_{L^2(\Omega)},
& v\in L^2(\Omega),\quad v_h=P_hv,
\end{cases}
\end{equation}
where $p=1$ for BE and $p=2$ for corrected SBD. The constant $C$ is independent of $h$ and $\tau$. Consequently, the estimates give first-order temporal convergence for BE and second-order convergence for corrected SBD at any fixed positive time, for both types of initial data.

\subsection{Direct convolution response tests}
\label{subsec:experiment-response}

We first test the direct convolution response
\[
y(t)=F(\partial_t)g_2(t),
\qquad
y_\tau^n=F(\partial_\tau)g_2(t_n),
\qquad g_2(t)=t.
\]
The exact response and the associated radial
factor are
\[
y(t)=\int_0^1
\frac{t^{q+1}}{\Gamma(q+2)}\,\mathrm dq,
\qquad
\Phi(t^{-1})
=\frac{1-t}{\log(1/t)},
\quad 0<t<1.
\]
By \eqref{eq:application-power-error},
the absolute error satisfies
\begin{equation}\label{eq:direct-response-error}
\mathcal E_\tau(t_n)
:=|y_\tau^n-y(t_n)|
\leq C\tau^pt_n^{1-p}\Phi(t_n^{-1}),
\qquad p=1,2.
\end{equation}
Table~\ref{tab:direct-cq-1} reports the errors
at $T=0.1$, with convergence rates computed by
\[
\mathrm{rate}
=
\frac{\log(\mathcal E_\tau(T)/
\mathcal E_{\tau/2}(T))}{\log2}.
\]
The observed rates approach one for BE and
two for SBD, in agreement with the theoretical
temporal orders.

\begin{table}[htbp]
\centering
\caption{Absolute errors for the direct CQ response
to $g_2(t)=t$ at $T=0.1$.}
\label{tab:direct-cq-1}
\setlength{\tabcolsep}{4mm}{
\begin{tabular}{ccccc}
\toprule
& \multicolumn{2}{c}{BE}
& \multicolumn{2}{c}{SBD} \\
\cmidrule(lr){2-3}\cmidrule(l){4-5}
$N$ & Error & Rate & Error & Rate \\
\midrule
5  & $1.35\times10^{-3}$ & --
   & $5.74\times10^{-5}$ & -- \\
10 & $6.83\times10^{-4}$ & 0.982
   & $1.33\times10^{-5}$ & 2.113 \\
20 & $3.43\times10^{-4}$ & 0.991
   & $3.22\times10^{-6}$ & 2.044 \\
40 & $1.72\times10^{-4}$ & 0.995
   & $7.93\times10^{-7}$ & 2.020 \\
80 & $8.63\times10^{-5}$ & 0.998
   & $1.97\times10^{-7}$ & 2.010 \\
\bottomrule
\end{tabular}}
\end{table}

To examine the short-time dependence, we use
SBD with $N=1000$ and $\tau=t/N$.
In this case,
\eqref{eq:direct-response-error} reduces to
\begin{equation}\label{eq:distributed-response-bound-comparison}
\mathcal E_\tau(t)
\leq CN^{-2}t\Phi(t^{-1})
=CN^{-2}t\frac{1-t}{\log(1/t)}.
\end{equation}
the radial estimate retains a logarithmic factor that is lost when $\Phi(t^{-1})$ is replaced by a constant. The latter replacement gives the coarser bound $CN^{-2}t$. Table~\ref{tab:distributed-response-short} compares the errors normalized by the radial scale $N^{-2}t\Phi(t^{-1})$ and the coarser
power scale $N^{-2}t$. The absolute errors decrease as $t$ approaches zero, while the errors normalized by $N^{-2}t\Phi(t^{-1})$ remain bounded and become nearly constant for small $t$. In contrast, the normalization by $N^{-2}t$ does not account for the additional logarithmic decay.

\begin{table}[htbp]
\centering
\caption{Short-time SBD errors for the distributed-order
response, with $N=1000$ and $\tau=t/N$.}
\label{tab:distributed-response-short}
\setlength{\tabcolsep}{3mm}
\begin{tabular}{cccc}
\toprule
$t$ & Absolute error &
$\mathcal E_\tau/(N^{-2}t\Phi(t^{-1}))$ &
$\mathcal E_\tau/(N^{-2}t)$ \\
\midrule
$10^{-3}$ & $2.74\times10^{-13}$
& $1.89\times10^{-3}$ & $2.74\times10^{-4}$ \\
$10^{-4}$ & $1.26\times10^{-14}$
& $1.16\times10^{-3}$ & $1.26\times10^{-4}$ \\
$10^{-5}$ & $1.96\times10^{-15}$
& $2.26\times10^{-3}$ & $1.96\times10^{-4}$ \\
$10^{-6}$ & $1.90\times10^{-16}$
& $2.62\times10^{-3}$ & $1.90\times10^{-4}$ \\
$10^{-7}$ & $1.67\times10^{-17}$
& $2.70\times10^{-3}$ & $1.67\times10^{-4}$ \\
$10^{-8}$ & $1.44\times10^{-18}$
& $2.65\times10^{-3}$ & $1.44\times10^{-4}$ \\
\bottomrule
\end{tabular}
\end{table}

\subsection{PDE tests}
\label{subsec:experiment-pde}

We next examine temporal convergence for the distributed-order evolution problem \eqref{eq:experiment-distributed-pde}. All computations use the fixed spatial mesh and initial projections specified above. The relative temporal error is defined by
\[
E_\tau(T)
=
\frac{\|u_h(T)-U_h^N\|_{L^2(\Omega)}}
{\|v\|_{L^2(\Omega)}}.
\]
Table~\ref{tab:distributed-time} shows first-order convergence for BE and second-order convergence for corrected SBD with both smooth and nonsmooth initial data, which are consistent with the theoretical error estimates in \eqref{eq:distributed-pde-time-error}.

\begin{table}[htbp]
\centering
\caption{Normalized temporal $L^2$ errors and rates
at $T=0.1$, with $h=2^{-11}$ and $\tau=T/N$.}
\label{tab:distributed-time}
\setlength{\tabcolsep}{4mm}{
\begin{tabular}{ccccc}
\toprule
& \multicolumn{2}{c}{BE}
& \multicolumn{2}{c}{SBD} \\
\cmidrule(lr){2-3}\cmidrule(l){4-5}
$N$ & Error & Rate & Error & Rate \\
\midrule
\multicolumn{5}{l}{(a) Smooth initial data} \\
5  & $2.11\times10^{-3}$ & --
   & $4.60\times10^{-4}$ & -- \\
10 & $1.04\times10^{-3}$ & 1.023
   & $9.82\times10^{-5}$ & 2.227 \\
20 & $5.15\times10^{-4}$ & 1.012
   & $2.30\times10^{-5}$ & 2.092 \\
40 & $2.56\times10^{-4}$ & 1.006
   & $5.59\times10^{-6}$ & 2.044 \\
80 & $1.28\times10^{-4}$ & 1.003
   & $1.38\times10^{-6}$ & 2.022 \\
\midrule
\multicolumn{5}{l}{(b) Nonsmooth initial data} \\
5  & $3.88\times10^{-3}$ & --
   & $7.77\times10^{-4}$ & -- \\
10 & $1.92\times10^{-3}$ & 1.015
   & $1.67\times10^{-4}$ & 2.218 \\
20 & $9.54\times10^{-4}$ & 1.008
   & $3.93\times10^{-5}$ & 2.087 \\
40 & $4.76\times10^{-4}$ & 1.004
   & $9.55\times10^{-6}$ & 2.042 \\
80 & $2.38\times10^{-4}$ & 1.002
   & $2.35\times10^{-6}$ & 2.020 \\
\bottomrule
\end{tabular}}
\end{table}

Finally, we investigate the short-time behavior
of the SBD scheme. Let $t_k=10^{-k}$ and
$\tau_k=t_k/N$, with $N=1000$. We define
\[
E(t_k)
=
\frac{\|u_h(t_k)-U_h^N\|_{L^2(\Omega)}}
{\|v\|_{L^2(\Omega)}},
\qquad
\eta_k
=
\frac{\log(E(t_{k-1})/E(t_k))}{\log10}.
\]
Table~\ref{tab:distributed-short} shows that the temporal
errors decrease slowly as $t$ and $\tau=t/N$ approach zero.
The corresponding $\eta_k$ also decrease. To interpret this
trend, note that if $E(t)\sim Ct^\alpha$ for some $\alpha>0$,
then $\eta_k\to\alpha$ as $k\to\infty$. In contrast,
for logarithmic decay of the form
$E(t)=C|\log t|^{-\sigma}$ with $\sigma>0$,
we have
\[
\eta_k
=\sigma\log_{10}\left(\frac{k}{k-1}\right)
\sim\frac{\sigma}{k\log 10}
\longrightarrow0.
\]
Thus, power decay is characterized by $\eta_k$ approaching a positive constant, whereas logarithmic decay leads to $\eta_k$ decreasing to zero. The decreasing values reported in Table~\ref{tab:distributed-short} are consistent with logarithmic short-time behavior.

\begin{table}[htbp]
\centering
\caption{Short-time SBD errors with $h=2^{-11}$
and $\tau=t/1000$.}
\label{tab:distributed-short}
\setlength{\tabcolsep}{4mm}{
\begin{tabular}{ccccc}
\toprule
& \multicolumn{2}{c}{(a) Smooth data}
& \multicolumn{2}{c}{(b) Nonsmooth data} \\
\cmidrule(lr){2-3}\cmidrule(l){4-5}
$t$ & Error & $\eta_k$ & Error & $\eta_k$ \\
\midrule
$10^{-3}$ & $7.23\times10^{-9}$ & --
& $9.13\times10^{-9}$ & -- \\
$10^{-4}$ & $6.58\times10^{-9}$ & 0.041
& $7.53\times10^{-9}$ & 0.084 \\
$10^{-5}$ & $6.01\times10^{-9}$ & 0.039
& $6.37\times10^{-9}$ & 0.073 \\
$10^{-6}$ & $5.52\times10^{-9}$ & 0.037
& $5.50\times10^{-9}$ & 0.064 \\
$10^{-7}$ & $5.08\times10^{-9}$ & 0.036
& $4.82\times10^{-9}$ & 0.057 \\
$10^{-8}$ & $4.70\times10^{-9}$ & 0.034
& $4.28\times10^{-9}$ & 0.052 \\
\bottomrule
\end{tabular}}
\end{table}

\section{Conclusion}\label{sec:conclusion}

We developed stability and error estimates for multistep convolution quadrature under positive continuous radial bounds that are regularly varying at zero and infinity. The resulting global comparison gives estimates valid at all positive grid times, with constants independent of a terminal time. The analysis applies to standard multistep CQ schemes and obtains the resulting estimates directly in terms of $\Phi(t^{-1})$. This formulation covers symbols that admit no global single-power bound on the sector and preserves logarithmic and multiscale factors that would be lost under a coarser power estimate.

Applications to distributed-order and logarithmic-power memory operators demonstrate how the radial assumption can be verified directly from the corresponding Laplace symbols. The response experiments recover the expected first- and second-order temporal convergence and illustrate the short-time behavior described by the radial error bounds. The PDE experiments further demonstrate temporal convergence for both smooth and nonsmooth initial data.

%\section*{Declarations}
%
%Some journals require declarations to be submitted in a standardised format. Please check the Instructions for Authors of the journal to which you are submitting to see if you need to complete this section. If yes, your manuscript must contain the following sections under the heading `Declarations':
%
%\begin{itemize}
%\item Funding
%\item Conflict of interest/Competing interests (check journal-specific guidelines for which heading to use)
%\item Ethics approval and consent to participate
%\item Consent for publication
%\item Data availability
%\item Materials availability
%\item Code availability
%\item Author contribution
%\end{itemize}
%
%\noindent
%If any of the sections are not relevant to your manuscript, please include the heading and write `Not applicable' for that section.

\begin{appendices}

%\section{}\label{secA1}

%%=============================================%%
%% For submissions to Nature Portfolio Journals %%
%% please use the heading ``Extended Data''.   %%
%%=============================================%%

%%=============================================================%%
%% Sample for another appendix section			       %%
%%=============================================================%%

%% \section{Example of another appendix section}\label{secA2}%
%% Appendices may be used for helpful, supporting or essential material that would otherwise
%% clutter, break up or be distracting to the text. Appendices can consist of sections, figures,
%% tables and equations etc.

\end{appendices}

%%===========================================================================================%%
%% If you are submitting to one of the Nature Portfolio journals, using the eJP submission   %%
%% system, please include the references within the manuscript file itself. You may do this  %%
%% by copying the reference list from your .bbl file, paste it into the main manuscript .tex %%
%% file, and delete the associated \verb+\bibliography+ commands.                            %%
%%===========================================================================================%%

\bibliography{sn-bibliography}% common bib file
%% if required, the content of .bbl file can be included here once bbl is generated
%%\input sn-article.bbl

\end{document}